\documentclass[12pt]{article}

\usepackage{amsmath}
\usepackage{geometry}
\usepackage{amsthm}
\usepackage{mathrsfs}
\usepackage{xcolor}
\usepackage{amssymb}

\newtheorem{theorem}{Theorem}[section]
\newtheorem{definition}[theorem]{Definition}

\newtheorem{proposition}[theorem]{Proposition}
\newtheorem{corollary}[theorem]{Corollary}

\newtheorem{remark}[theorem]{Remark}

\title{Gauge-compatible tensors on statistical manifolds: splitting and submanifold geometry}
\author{
Mirjana Milijevi\'c\\
Faculty of Economics\\
University of Banja Luka\\
\texttt{mirjana.milijevic@ef.unibl.org}
\and
Luis P. Yapu\\
Instituto de Matem\'atica e Estat\'istica \\
Universidade Federal Fluminense\\
and Chair of Dynamics, Control, Machine Learning and Numerics\\
Friedrich-Alexander-Universit\"at Erlangen-N\"urnberg\\
\texttt{luis.yapu@gmail.com}
}
\date{}

\begin{document}

\maketitle
\begin{abstract}
We study statistical manifolds $(M,g,\nabla,\nabla^{*})$ endowed with a nonzero $(1,1)$-tensor field $\Theta$ satisfying the gauge equation
\[
\nabla_X(\Theta Y)=\Theta(\nabla_X^{*}Y).
\]
We first characterize this condition in terms of the statistical difference tensor
\[
K=\nabla-\nabla^{g}.
\]
When $\Theta$ is parallel with respect to the Levi-Civita connection, the gauge equation is equivalent to
\[
K_X\Theta=-\Theta K_X.
\]
We further show that $\Theta$ intertwines the parallel transports of the dual connections. Consequently, its rank is constant on every connected component, and $\ker\Theta$ and $\operatorname{Im}\Theta$ determine smooth integrable distributions. If, in addition,
\[
TM=\ker\Theta\overset{\perp}{\oplus}\operatorname{Im}\Theta,
\]
we establish a local product decomposition of the statistical structure. We then study submanifolds carrying $\Theta$-invariant and $\Theta$-anti-invariant distributions and derive the tangential and normal components of the ambient gauge equation in terms of the second fundamental forms and shape operators of the dual statistical connections. We also obtain curvature-intertwining consequences and present a non-totally-geodesic example illustrating the submanifold identities.
\end{abstract}



\section{Introduction}

Statistical manifolds provide a differential-geometric framework for the study of dual affine structures arising in information geometry, affine differential geometry, and related areas. A statistical manifold is a Riemannian manifold $(M,g)$ equipped with a torsion-free affine connection $\nabla$ such that $\nabla g$ is totally symmetric. Equivalently, the statistical structure may be described by the dual pair $(\nabla,\nabla^{*})$, where
\[
Xg(Y,Z)
=
g(\nabla_XY,Z)+g(Y,\nabla_X^{*}Z).
\]
The Levi-Civita connection $\nabla^{g}$ is the midpoint of the dual connections, and one writes
\[
\nabla=\nabla^{g}+K,
\qquad
\nabla^{*}=\nabla^{g}-K,
\]
where the statistical difference tensor $K$ is symmetric and the cubic form
\[
C(X,Y,Z)=g(K_XY,Z)
\]
is totally symmetric. We refer to Amari \cite{amari-85}, Lauritzen \cite{lauritzen-87}, and Noguchi \cite{noguchi-92} for the basic theory of statistical manifolds.

An important problem in statistical geometry is to understand statistical structures that are compatible with an additional geometric tensor. The best-known example is the holomorphic statistical case, in which the underlying Riemannian manifold is K\"{a}hler and the statistical difference tensor is compatible with the complex structure. Related constructions occur for tensors associated with almost contact and coK\"{a}hler geometry. These examples indicate that many compatibility properties do not depend on all the algebraic identities satisfied by a complex or contact structure, but rather on an intertwining relation between a tensor field and the dual affine connections.

Motivated by this observation, we consider a nonzero $(1,1)$-tensor field $\Theta$ satisfying the gauge equation
\begin{equation}\label{eq:gauge-intro}
\nabla_X(\Theta Y)=\Theta(\nabla_X^{*}Y),
\qquad
X,Y\in\Gamma(TM).
\end{equation}
Thus, $\Theta$ is a morphism from the vector bundle with connection $(TM,\nabla^{*})$ to $(TM,\nabla)$. Statistical manifolds endowed with such a tensor will be called
$\Theta$-statistical manifolds. The gauge-compatible tensor viewpoint
was developed in~\cite{milijevic-puechmorel}. In that work, the main
emphasis was on the algebraic and homological structures associated
with solutions of the gauge equation, in particular on
Koszul--Vinberg (pre-Lie) structures, obstruction tensors, and
spectral-sequence degeneracy. Some basic geometric consequences of the
gauge equation, including the constant-rank property and the
integrability of $\ker\Theta$ and $\operatorname{Im}\Theta$, were also
established there.

The aim of the present paper is different. We focus on the
differential-geometric consequences of the gauge equation that arise
directly from the interaction between the dual affine connections and
the tensor $\Theta$. More precisely, we show that $\Theta$ intertwines
the parallel transports of $\nabla^{*}$ and $\nabla$, derive the
corresponding curvature and holonomy relations, and prove a local
statistical product decomposition for Levi-Civita-parallel gauge
tensors whose kernel and image are orthogonal complements. We then
develop a submanifold theory for general gauge-compatible tensors by
deriving the tangential and normal components of the ambient gauge
equation and obtaining integrability and autoparallelness criteria for
the invariant and anti-invariant distributions. Thus, while
\cite{milijevic-puechmorel} is primarily concerned with the algebraic
and homological structures induced by the gauge equation, the present
work is centered on its intrinsic and extrinsic differential geometry.
Our first observation is that the gauge equation admits a useful expression in terms of the statistical difference tensor:
\begin{equation}\label{eq:K-intro}
(\nabla_X^{g}\Theta)Y
+
K_X(\Theta Y)
+
\Theta K_XY
= 0.
\end{equation}
In particular, if $\Theta$ is parallel with respect to the Levi-Civita connection, then \eqref{eq:K-intro} reduces to
\[
K_X\Theta=-\Theta K_X.
\]
This identity describes the interaction between the statistical structure and the algebraic properties of $\Theta$. It recovers the familiar compatibility condition in the complex case and also applies to singular tensors, including projection and coK\"{a}hler-type tensors.

The connection-morphism interpretation of \eqref{eq:gauge-intro} also yields global information along curves. We prove that $\Theta$ intertwines the parallel transports of $\nabla^{*}$ and $\nabla$. It follows that the rank of $\Theta$ is constant on every connected component of $M$. Consequently, $\ker\Theta$ and $\operatorname{Im}\Theta$
are smooth distributions. The gauge equation further implies that $\operatorname{Im}\Theta$ is preserved by $\nabla$, whereas $\ker\Theta$ is preserved by $\nabla^{*}$; since both connections are torsion-free, these distributions are integrable.

For Levi-Civita-parallel tensors, the kernel and image may lead to a stronger geometric decomposition. Assuming that
\[
TM
=
\ker\Theta
\overset{\perp}{\oplus}
\operatorname{Im}\Theta,
\]
we show that both distributions are preserved by $\nabla^{g}$, $\nabla$, and $\nabla^{*}$. Moreover, the mixed components of the statistical difference tensor vanish. The local de Rham decomposition therefore induces a local product decomposition not only of the Riemannian metric but of the complete statistical structure. This gives, in particular, a local splitting result for gauge-compatible coK\"{a}hler statistical structures.

The second part of the paper concerns submanifold geometry. Let $M$ be isometrically immersed in a $\Theta$-statistical manifold $\widetilde M$. Along $M$, we decompose the action of $\Theta$ on tangent vector fields as
\[
\Theta Y=P_{\Theta}Y+F_{\Theta}Y,
\]
where $P_{\Theta}Y$ and $F_{\Theta}Y$ denote the tangential and normal components, respectively. We derive the tangential and normal components of the ambient gauge equation in terms of the second fundamental forms and shape operators associated with the dual statistical connections. These formulas describe the obstruction for the tangential and normal parts of $\Theta$ to satisfy induced gauge-type equations on the submanifold.

We then consider submanifolds whose tangent bundle admits an orthogonal decomposition
\[
TM=D_{\Theta}\oplus D_{\Theta}^{\perp},
\]
where $D_{\Theta}$ is $\Theta$-invariant and $\Theta(D_{\Theta}^{\perp})$ is contained in the normal bundle. This extends the invariant/anti-invariant decomposition appearing in the classical theory of CR-submanifolds of almost Hermitian manifolds; see, for example, \cite{djoric-okumura-09}. The corresponding CR geometry in the statistical setting has been
studied for holomorphic statistical manifolds in
\cite{MilijevicCR,SiddiquiEtAl}, where the interaction between the
invariant and anti-invariant distributions and the dual statistical
connections was investigated. More recently, Levi-flat CR statistical
submanifolds of maximal CR dimension were considered in
\cite{MilijevicMiri}. The present approach differs from these works in
that the ambient complex structure is replaced by an arbitrary
statistical gauge tensor $\Theta$. Thus the invariant/anti-invariant
decomposition is studied independently of the algebraic identity
$\Theta^{2}=-I$ and applies, in particular, to singular
gauge-compatible tensors. The general submanifold identities specialize naturally to the two distributions and relate their induced gauge behavior to the extrinsic geometry of the immersion. A non-totally-geodesic proper CR-type example is included to show that the resulting shape-operator and second-fundamental-form terms need not vanish.

Finally, the gauge equation implies the curvature-intertwining identity
\[
R^{\nabla}(X,Y)\Theta Z
=
\Theta R^{\nabla^{*}}(X,Y)Z,
\]
which is the infinitesimal counterpart of the parallel-transport relation. When $\Theta$ is invertible, the curvature operators and holonomy representations of the dual connections are conjugate through $\Theta$.

The paper is organized as follows. In Section~2 we establish the basic properties of statistical gauge tensors, including the parallel-transport relation, the characterization in terms of the difference tensor, the curvature identity, and the local splitting theorem. Section~3 is devoted to submanifolds of $\Theta$-statistical manifolds and to the induced tangential and normal gauge equations, with particular attention to invariant and anti-invariant distributions. In Section~4 we present examples illustrating both the intrinsic and extrinsic constructions.

\section{Statistical gauge tensors}\label{sec:gauge-tensors}

Let $(M,g,\nabla,\nabla^{*})$ be a statistical manifold. Thus, $\nabla$ and $\nabla^{*}$ are torsion-free affine connections dual with respect to $g$:
\[
Xg(Y,Z)
=
g(\nabla_XY,Z)+g(Y,\nabla_X^{*}Z).
\]
If $\nabla^{g}$ denotes the Levi-Civita connection of $g$, then
\begin{equation}\label{eq:dual-connections}
\nabla=\nabla^{g}+K,
\qquad
\nabla^{*}=\nabla^{g}-K,
\end{equation}
where the statistical difference tensor $K$ satisfies
\begin{equation}\label{eq:K-properties}
K_XY=K_YX,
\qquad
g(K_XY,Z)=g(Y,K_XZ).
\end{equation}
Equivalently, the cubic form
\[
C(X,Y,Z)=g(K_XY,Z)
\]
is totally symmetric.

\begin{definition}\label{def:theta-statistical}
A nonzero $(1,1)$-tensor field $\Theta$ on $M$ is called a statistical gauge tensor if
\begin{equation}\label{eq:gauge}
\nabla_X(\Theta Y)=\Theta(\nabla_X^{*}Y),
\qquad
X,Y\in\Gamma(TM).
\end{equation}
A statistical manifold endowed with such a tensor will be called a
$\Theta$-statistical manifold.
\end{definition}

Equation \eqref{eq:gauge} means that
\[
\Theta:(TM,\nabla^{*})\longrightarrow(TM,\nabla)
\]
is a morphism of vector bundles with connection. This interpretation gives an immediate description of the behavior of $\Theta$ under parallel transport.

\subsection{Parallel transport, rank, and curvature}

Let $\gamma:[a,b]\to M$ be a piecewise smooth curve. For
$s,t\in[a,b]$, let
\[
\mathcal{P}^{\nabla}_{\gamma;s,t}
  :T_{\gamma(s)}M\longrightarrow T_{\gamma(t)}M,
\qquad
\mathcal{P}^{\nabla^{*}}_{\gamma;s,t}
  :T_{\gamma(s)}M\longrightarrow T_{\gamma(t)}M
\]
denote the parallel transport maps along $\gamma$ with respect to
$\nabla$ and $\nabla^{*}$, respectively.

\begin{theorem}\label{thm:parallel-transport}
Let $(M,g,\nabla,\nabla^{*},\Theta)$ be a $\Theta$-statistical
manifold. Then, for every piecewise smooth curve
$\gamma:[a,b]\to M$ and every $s,t\in[a,b]$,
\[
\mathcal{P}^{\nabla}_{\gamma;s,t}\circ\Theta_{\gamma(s)}
=
\Theta_{\gamma(t)}
\circ
\mathcal{P}^{\nabla^{*}}_{\gamma;s,t}.
\]
\end{theorem}

\begin{proof}
Let $v\in T_{\gamma(s)}M$, and let $Y$ be the $\nabla^{*}$-parallel vector field along $\gamma$ satisfying $Y(s)=v$. By the gauge equation,
\[
\nabla_{\dot\gamma}(\Theta Y)
=
\Theta(\nabla_{\dot\gamma}^{*}Y).
\]
Thus $\Theta Y$ is $\nabla$-parallel. Consequently,
\[
\Theta_{\gamma(t)}
\left(
\mathcal{P}^{\nabla^{*}}_{\gamma;s,t}v
\right)
=
\mathcal{P}^{\nabla}_{\gamma;s,t}
\left(
\Theta_{\gamma(s)}v
\right),
\]
which proves the assertion.
\end{proof}

The constant-rank and integrability properties of gauge-compatible
tensors were established in \cite{milijevic-puechmorel}.

\begin{corollary}
On each connected component of $M$, the tensor $\Theta$ has constant
rank. Moreover, $\operatorname{Im}\Theta$ is preserved by $\nabla$,
$\ker\Theta$ is preserved by $\nabla^{*}$, and both distributions are
integrable.
\end{corollary}

The curvature identity corresponding to Theorem~\ref{thm:parallel-transport} is obtained by applying the gauge equation twice.

\begin{proposition}\label{prop:curvature-intertwining}
Let $(M,g,\nabla,\nabla^{*},\Theta)$ be a $\Theta$-statistical manifold. Then
\[
R^{\nabla}(X,Y)\Theta Z
=
\Theta\,R^{\nabla^{*}}(X,Y)Z,
\]
for all $X,Y,Z\in\Gamma(TM)$.

If $\Theta$ is invertible, then
\[
R^{\nabla}(X,Y)
=
\Theta\circ R^{\nabla^{*}}(X,Y)\circ\Theta^{-1}.
\]

Moreover, for every piecewise smooth loop $\gamma$ based at $p$,
\[
\mathcal{P}^{\nabla}_{\gamma}\circ\Theta_p
=
\Theta_p\circ\mathcal{P}^{\nabla^{*}}_{\gamma}.
\]
Hence, if $\Theta_p$ is invertible,
\[
\operatorname{Hol}_{p}(\nabla)
=
\Theta_p\circ
\operatorname{Hol}_{p}(\nabla^{*})
\circ\Theta_p^{-1}.
\]
\end{proposition}

\begin{proof}
Using the gauge equation, we have
\[
\nabla_Y(\Theta Z)
=
\Theta(\nabla_Y^{*}Z).
\]
Therefore,
\[
\nabla_X\nabla_Y(\Theta Z)
=
\nabla_X\!\left(\Theta(\nabla_Y^{*}Z)\right)
=
\Theta\!\left(\nabla_X^{*}\nabla_Y^{*}Z\right).
\]
Similarly,
\[
\nabla_Y\nabla_X(\Theta Z)
=
\Theta\!\left(\nabla_Y^{*}\nabla_X^{*}Z\right),
\]
and
\[
\nabla_{[X,Y]}(\Theta Z)
=
\Theta\!\left(\nabla_{[X,Y]}^{*}Z\right).
\]
Subtracting these identities gives
\[
R^{\nabla}(X,Y)\Theta Z
=
\Theta\,R^{\nabla^{*}}(X,Y)Z.
\]

If $\Theta$ is invertible, the curvature conjugacy formula follows immediately.
The parallel-transport identity for loops follows from
Theorem~\ref{thm:parallel-transport}, and the holonomy relation follows by
conjugation with $\Theta_p$.
\end{proof}

\subsection{The statistical difference tensor}

The gauge equation can be expressed entirely in terms of the Levi-Civita connection and the statistical difference tensor.

\begin{proposition}\label{prop:K-characterization}
Let $(M,g,\nabla,\nabla^{*})$ be a statistical manifold, with $\nabla$ and $\nabla^{*}$ written as in \eqref{eq:dual-connections}. A $(1,1)$-tensor field $\Theta$ satisfies the gauge equation if and only if
\begin{equation}\label{eq:K-characterization}
(\nabla_X^{g}\Theta)Y
+
K_X(\Theta Y)
+
\Theta K_XY
=
0.
\end{equation}
In particular, if
\[
\nabla^{g}\Theta=0,
\]
then the gauge equation is equivalent to
\begin{equation}\label{eq:anticommutation}
K_X\Theta=-\Theta K_X
\end{equation}
for every $X\in\Gamma(TM)$.
\end{proposition}

\begin{proof}
By \eqref{eq:dual-connections},
\[
\nabla_X(\Theta Y)
=
\nabla_X^{g}(\Theta Y)+K_X(\Theta Y),
\]
whereas
\[
\Theta(\nabla_X^{*}Y)
=
\Theta(\nabla_X^{g}Y)-\Theta K_XY.
\]
Since
\[
\nabla_X^{g}(\Theta Y)
=
(\nabla_X^{g}\Theta)Y+\Theta(\nabla_X^{g}Y),
\]
the gauge equation is equivalent to \eqref{eq:K-characterization}. If
$\nabla^{g}\Theta=0$, this reduces to \eqref{eq:anticommutation}.
\end{proof}

We call a $\Theta$-statistical manifold \emph{Levi-Civita parallel} when
$\nabla^{g}\Theta=0$. In this case, the algebraic properties of $\Theta$ impose corresponding restrictions on $K$. For example, when $\Theta^{2}=-I$, each endomorphism $K_X$ is anti-complex-linear. When $\Theta^{2}=I$, $K_X$ exchanges the two eigendistributions of $\Theta$. If $\Theta^{2}=\Theta$, then
\[
K_X(\operatorname{Im}\Theta)=0,
\qquad
K_X(\ker\Theta)\subset\ker\Theta.
\]

\subsection{A local splitting theorem}

We now consider Levi-Civita-parallel gauge tensors for which the kernel and image are orthogonal complements.

\begin{theorem}\label{thm:local-splitting}
Let $(M,g,\nabla,\nabla^{*},\Theta)$ be a Levi-Civita-parallel
$\Theta$-statistical manifold. Assume that the tangent bundle admits the orthogonal decomposition
\[
TM=\ker\Theta\oplus\operatorname{Im}\Theta.
\]
Then both distributions are preserved by
\[
\nabla^{g},\ \nabla,\ \nabla^{*},
\]
and
\begin{equation}\label{eq:mixed-K-vanishing}
K_XY=0
\end{equation}
whenever one of $X,Y$ belongs to $\ker\Theta$ and the other belongs to
$\operatorname{Im}\Theta$.

Consequently, every point of $M$ admits a neighborhood $U$ for which
\[
(U,g,\nabla,\nabla^{*})
\cong
(U_0,g_0,\nabla^{0},(\nabla^{0})^{*})
\times
(U_1,g_1,\nabla^{1},(\nabla^{1})^{*}),
\]
where
\[
TU_0=\ker\Theta,
\qquad
TU_1=\operatorname{Im}\Theta.
\]
Thus the two local factors inherit statistical structures.
\end{theorem}

\begin{proof}
Let $Y\in\Gamma(\ker\Theta)$. Since $\nabla^{g}\Theta=0$,
\[
\Theta(\nabla_X^{g}Y)
=
\nabla_X^{g}(\Theta Y)
,
\]
so $\ker\Theta$ is preserved by $\nabla^{g}$. If
$Y\in\Gamma(\operatorname{Im}\Theta)$, write locally $Y=\Theta Z$. Then
\[
\nabla_X^{g}Y
=
\nabla_X^{g}(\Theta Z)=
\Theta(\nabla_X^{g}Z),
\]
and therefore $\operatorname{Im}\Theta$ is also preserved by $\nabla^{g}$.

The relation \eqref{eq:anticommutation} shows that $K$ preserves both distributions in its second argument. Indeed, if $Y\in\ker\Theta$, then
\[
0=K_X(\Theta Y)=-\Theta K_XY,
\]
and hence $K_XY\in\ker\Theta$. If $Y=\Theta Z\in\operatorname{Im}\Theta$, then
\[
K_XY
=
K_X(\Theta Z)=
-\Theta K_XZ
\in
\operatorname{Im}\Theta.
\]
It follows from \eqref{eq:dual-connections} that both distributions are preserved by $\nabla$ and $\nabla^{*}$.

Now let
\[
Y\in\Gamma(\ker\Theta),
\qquad
Z\in\Gamma(\operatorname{Im}\Theta).
\]
The preceding invariance properties imply
\[
K_YZ\in\Gamma(\operatorname{Im}\Theta).
\]
On the other hand, by the symmetry of $K$,
\[
K_YZ=K_ZY\in\Gamma(\ker\Theta).
\]
Therefore
\[
K_YZ=0.
\]

The two distributions are orthogonal and $\nabla^{g}$-parallel. The local de Rham decomposition theorem (\cite{KN-69}) therefore gives a local Riemannian product
\[
(U,g)\cong(U_0,g_0)\times(U_1,g_1).
\]
The preservation of the two distributions and the vanishing of the mixed components of $K$ show that $\nabla$ and $\nabla^{*}$ are product connections. Their restrictions to the two factors are torsion-free and dual with respect to the induced metrics, and therefore define statistical structures.
\end{proof}

The splitting theorem applies naturally to coK\"{a}hler statistical structures.

\begin{corollary}\label{cor:cokahler-splitting}
Let $(M^{2n+1},\phi,\xi,\eta,g)$ be a coK\"{a}hler manifold, and suppose that
$(M,g,\nabla,\nabla^{*},\phi)$ is a $\phi$-statistical manifold. Then $M$ is locally a statistical product
\[
(I,dt^{2},\nabla^{I},(\nabla^{I})^{*})
\times
(N^{2n},g_N,\nabla^{N},(\nabla^{N})^{*}),
\]
where
\[
TI=\ker\phi=\langle\xi\rangle,
\qquad
TN=\operatorname{Im}\phi=\ker\eta.
\]
Moreover, $J=\phi|_{TN}$ is a parallel complex structure and the induced statistical structure on $N$ is holomorphic.
\end{corollary}

\begin{proof}
On a coK\"{a}hler manifold,
\[
\nabla^{g}\phi=0,
\qquad
\ker\phi=\langle\xi\rangle,
\qquad
\operatorname{Im}\phi=\ker\eta,
\]
and these distributions are orthogonal complements. Theorem~\ref{thm:local-splitting} therefore applies. On the even-dimensional factor,
\[
J^{2}=-I,
\qquad
\nabla^{g_N}J=0,
\]
and the restriction of \eqref{eq:anticommutation} gives
\[
K_XJY=-JK_XY.
\]
Hence the induced structure on $N$ is holomorphic statistical.
\end{proof}
\section{Submanifold geometry}\label{sec:submanifolds}

Let
\[
(\widetilde M,\widetilde g,
 \widetilde\nabla,\widetilde\nabla^{*},\Theta)
\]
be a $\Theta$-statistical manifold and let
\[
M\hookrightarrow \widetilde M
\]
be an isometrically immersed submanifold. We denote by $g$ the induced
metric and by $\nabla,\nabla^{*}$ the affine connections induced on $M$.
For tangent vector fields $X,Y\in\Gamma(TM)$ and a normal vector field
$\nu\in\Gamma(T^{\perp}M)$, the statistical Gauss--Weingarten formulas are (see, for example, \cite{Vos-98})
\begin{align}
\widetilde\nabla_XY
 &=
 \nabla_XY+h(X,Y),                                      \label{eq:Gauss}\\
\widetilde\nabla^{*}_XY
 &=
 \nabla^{*}_XY+h^{*}(X,Y),                              \label{eq:Gauss-dual}\\
\widetilde\nabla_X\nu
 &=
 -A_{\nu}X+\nabla^{\perp}_X\nu.                         \label{eq:Weingarten}
\end{align}
Here $h$ and $h^{*}$ are the second fundamental forms associated with
$\widetilde\nabla$ and $\widetilde\nabla^{*}$, respectively.

For completeness, the dual Weingarten formula is
\[
\widetilde\nabla^{*}_{X}\nu
   =-A^{*}_{\nu}X+\nabla^{*\perp}_{X}\nu.
\]
By the duality of $\widetilde\nabla$ and
$\widetilde\nabla^{*}$,
\[
g(A_{\nu}X,Y)
   =\widetilde g(h^{*}(X,Y),\nu),
\qquad
g(A^{*}_{\nu}X,Y)
   =\widetilde g(h(X,Y),\nu).
\]

Along $M$, decompose $\Theta$ as
\begin{equation}\label{eq:Theta-tangent}
\Theta Y=P_{\Theta}Y+F_{\Theta}Y,
\qquad
P_{\Theta}Y\in\Gamma(TM),\quad
F_{\Theta}Y\in\Gamma(T^{\perp}M),
\end{equation}
for $Y\in\Gamma(TM)$, and
\begin{equation}\label{eq:Theta-normal}
\Theta\nu=t_{\Theta}\nu+f_{\Theta}\nu,
\qquad
t_{\Theta}\nu\in\Gamma(TM),\quad
f_{\Theta}\nu\in\Gamma(T^{\perp}M),
\end{equation}
for $\nu\in\Gamma(T^{\perp}M)$.

Applying the ambient gauge equation
\[
\widetilde\nabla_X(\Theta Y)
=
\Theta(\widetilde\nabla_X^{*}Y)
\]
and separating tangent and normal components gives the following
identities.

\begin{proposition}\label{prop:submanifold-gauge}
For every $X,Y\in\Gamma(TM)$,
\begin{align}
\nabla_X(P_{\Theta}Y)
-A_{F_{\Theta}Y}X
&=
P_{\Theta}(\nabla_X^{*}Y)
+t_{\Theta}(h^{*}(X,Y)),                 \label{eq:tangential-gauge}\\
h(X,P_{\Theta}Y)
+\nabla_X^{\perp}(F_{\Theta}Y)
&=
F_{\Theta}(\nabla_X^{*}Y)
+f_{\Theta}(h^{*}(X,Y)).                 \label{eq:normal-gauge}
\end{align}
\end{proposition}

\begin{proof}
Using \eqref{eq:Theta-tangent} and the Gauss--Weingarten formulas,
\[
\widetilde\nabla_X(\Theta Y)
=
\nabla_X(P_{\Theta}Y)
+h(X,P_{\Theta}Y)
-A_{F_{\Theta}Y}X
+\nabla_X^{\perp}(F_{\Theta}Y).
\]
On the other hand, by \eqref{eq:Gauss-dual},
\[
\Theta(\widetilde\nabla_X^{*}Y)
=
\Theta(\nabla_X^{*}Y)
+\Theta(h^{*}(X,Y)).
\]
Using \eqref{eq:Theta-tangent} and \eqref{eq:Theta-normal} on the
right-hand side and comparing tangent and normal components gives
\eqref{eq:tangential-gauge} and \eqref{eq:normal-gauge}.
\end{proof}

The identities above hold for every submanifold of a
$\Theta$-statistical manifold. We now specialize them to submanifolds
carrying invariant and anti-invariant tangent distributions. For the coK\"ahler case, related obstruction formulas were considered
in~\cite{milijevic-puechmorel}. Here we do not impose the vanishing of
those obstructions; instead, we study the invariant and anti-invariant
distributions for a general statistical gauge tensor and characterize
their integrability and autoparallelness in terms of the extrinsic
geometry of the immersion.

\begin{definition}\label{def:Theta-CR}
A submanifold $M$ of a $\Theta$-statistical manifold is called a
$\Theta$-\emph{CR submanifold} if its tangent bundle admits an
orthogonal decomposition
\begin{equation}\label{eq:Theta-CR-splitting}
TM=D_{\Theta}\oplus D_{\Theta}^{\perp}
\end{equation}
such that
\[
\Theta(D_{\Theta})\subset D_{\Theta},
\qquad
\Theta(D_{\Theta}^{\perp})\subset T^{\perp}M.
\]
The distributions $D_{\Theta}$ and $D_{\Theta}^{\perp}$ are called,
respectively, the invariant and anti-invariant distributions.
The submanifold is called proper when both distributions are nonzero.
\end{definition}
\begin{remark}
The terminology used here is deliberately geometric. In the present
paper, a $\Theta$-CR submanifold is defined by the orthogonal
invariant/anti-invariant decomposition \eqref{eq:Theta-CR-splitting}
alone. This is closer to the classical notion of a CR submanifold.
In~\cite{milijevic-puechmorel}, a more restrictive gauge-compatible
notion was considered, in which additional tangential and normal
obstruction tensors were required to vanish. The latter class is
therefore a distinguished gauge-flat subclass of the $\Theta$-CR
submanifolds considered here.
\end{remark}

Thus,
\begin{equation}\label{eq:CR-components}
Y\in D_{\Theta}
\quad\Longrightarrow\quad
P_{\Theta}Y=\Theta Y,\quad F_{\Theta}Y=0,
\end{equation}
whereas
\begin{equation}\label{eq:CR-components-perp}
U\in D_{\Theta}^{\perp}
\quad\Longrightarrow\quad
P_{\Theta}U=0,\quad F_{\Theta}U=\Theta U.
\end{equation}

The next result gives intrinsic--extrinsic identities governing the
integrability of the two distributions.

\begin{theorem}\label{thm:CR-integrability}
Let $M$ be a $\Theta$-CR submanifold of a $\Theta$-statistical
manifold.

\begin{enumerate}
\item For $X,Y\in\Gamma(D_{\Theta})$,
\begin{equation}\label{eq:D-integrability}
F_{\Theta}([X,Y])
=
h(X,\Theta Y)-h(Y,\Theta X).
\end{equation}

\item For $U,V\in\Gamma(D_{\Theta}^{\perp})$,
\begin{equation}\label{eq:Dperp-integrability}
P_{\Theta}([U,V])
=
A_{\Theta U}V-A_{\Theta V}U.
\end{equation}
\end{enumerate}

Consequently:

\begin{enumerate}
\item[(a)] if
\[
\ker(F_{\Theta}|_{TM})=D_{\Theta},
\]
then $D_{\Theta}$ is integrable if and only if
\begin{equation}\label{eq:D-integrability-condition}
h(X,\Theta Y)=h(Y,\Theta X)
\end{equation}
for all $X,Y\in\Gamma(D_{\Theta})$;

\item[(b)] if
\[
\ker(P_{\Theta}|_{TM})=D_{\Theta}^{\perp},
\]
then $D_{\Theta}^{\perp}$ is integrable if and only if
\begin{equation}\label{eq:Dperp-integrability-condition}
A_{\Theta U}V=A_{\Theta V}U
\end{equation}
for all $U,V\in\Gamma(D_{\Theta}^{\perp})$.
\end{enumerate}
\end{theorem}

\begin{proof}
Let $X,Y\in\Gamma(D_{\Theta})$. Since
$F_{\Theta}X=F_{\Theta}Y=0$ and
$P_{\Theta}X=\Theta X$, $P_{\Theta}Y=\Theta Y$,
equation \eqref{eq:normal-gauge} gives
\[
h(X,\Theta Y)
=
F_{\Theta}(\nabla_X^{*}Y)
+
f_{\Theta}(h^{*}(X,Y)).
\]
Interchanging $X$ and $Y$ and subtracting, the symmetry of $h^{*}$
yields
\[
h(X,\Theta Y)-h(Y,\Theta X)
=
F_{\Theta}
\bigl(
\nabla_X^{*}Y-\nabla_Y^{*}X
\bigr).
\]
Since $\nabla^{*}$ is torsion-free,
\[
\nabla_X^{*}Y-\nabla_Y^{*}X=[X,Y],
\]
and \eqref{eq:D-integrability} follows.

Now let $U,V\in\Gamma(D_{\Theta}^{\perp})$. Then
$P_{\Theta}U=P_{\Theta}V=0$ and
$F_{\Theta}U=\Theta U$, $F_{\Theta}V=\Theta V$.
Equation \eqref{eq:tangential-gauge} gives
\[
-A_{\Theta V}U
=
P_{\Theta}(\nabla_U^{*}V)
+t_{\Theta}(h^{*}(U,V)).
\]
Interchanging $U$ and $V$ and subtracting again eliminates the
symmetric $h^{*}$-term and gives
\[
A_{\Theta U}V-A_{\Theta V}U
=
P_{\Theta}
\bigl(
\nabla_U^{*}V-\nabla_V^{*}U
\bigr).
\]
Torsion-freeness of $\nabla^{*}$ yields
\eqref{eq:Dperp-integrability}.

The final assertions follow because a distribution is integrable
precisely when the bracket of any two of its sections is again a
section of the same distribution. Under the stated kernel assumptions,
\[
F_{\Theta}([X,Y])=0
\quad\Longleftrightarrow\quad
[X,Y]\in D_{\Theta},
\]
and
\[
P_{\Theta}([U,V])=0
\quad\Longleftrightarrow\quad
[U,V]\in D_{\Theta}^{\perp}.
\]
\end{proof}

\begin{remark}\label{rem:kernel-assumptions}
The kernel assumptions in Theorem~\ref{thm:CR-integrability} are
automatic in several standard situations. For example,
\[
\ker(F_{\Theta}|_{TM})=D_{\Theta}
\]
whenever $\Theta$ is injective on $D_{\Theta}^{\perp}$, while
\[
\ker(P_{\Theta}|_{TM})=D_{\Theta}^{\perp}
\]
whenever $\Theta$ is injective on $D_{\Theta}$.
They are included explicitly because a general statistical gauge
tensor may be singular.
\end{remark}

The same formulas also characterize when the two distributions are
autoparallel with respect to the induced dual connection.

\begin{proposition}\label{prop:CR-autoparallel}
Let $M$ be a $\Theta$-CR submanifold.

\begin{enumerate}
\item Assume
\[
\ker(F_{\Theta}|_{TM})=D_{\Theta}.
\]
Then $D_{\Theta}$ is $\nabla^{*}$-autoparallel if and only if
\begin{equation}\label{eq:D-autoparallel}
h(X,\Theta Y)
=
f_{\Theta}(h^{*}(X,Y))
\end{equation}
for all $X,Y\in\Gamma(D_{\Theta})$.

\item Assume
\[
\ker(P_{\Theta}|_{TM})=D_{\Theta}^{\perp}.
\]
Then $D_{\Theta}^{\perp}$ is $\nabla^{*}$-autoparallel if and only if
\begin{equation}\label{eq:Dperp-autoparallel}
A_{\Theta V}U
+
t_{\Theta}(h^{*}(U,V))
=
0
\end{equation}
for all $U,V\in\Gamma(D_{\Theta}^{\perp})$.
\end{enumerate}
\end{proposition}

\begin{proof}
For $X,Y\in\Gamma(D_{\Theta})$, equation
\eqref{eq:normal-gauge} reduces to
\[
F_{\Theta}(\nabla_X^{*}Y)
=
h(X,\Theta Y)
-
f_{\Theta}(h^{*}(X,Y)).
\]
Under the first kernel assumption,
$\nabla_X^{*}Y\in D_{\Theta}$ if and only if the left-hand side
vanishes, proving the first assertion.

For $U,V\in\Gamma(D_{\Theta}^{\perp})$,
equation \eqref{eq:tangential-gauge} becomes
\[
P_{\Theta}(\nabla_U^{*}V)
=
-A_{\Theta V}U
-t_{\Theta}(h^{*}(U,V)).
\]
The second kernel assumption gives the desired equivalence.
\end{proof}
\section{Examples}\label{sec:examples}

We conclude with two examples illustrating the intrinsic and extrinsic
constructions developed above. The first gives a family of
gauge-compatible statistical structures on an oriented Riemannian
surface. The second gives a proper $\Theta$-CR submanifold which is
not totally geodesic.

\subsection{A family of gauge-compatible structures on Riemannian surfaces}

Let $(M^2,g)$ be an oriented Riemannian surface and let
$\Omega\subset M$ be an open subset admitting a positively oriented
orthonormal frame $\{U,V\}$. Denote by
$u$ and $v$ the dual $1$-forms:
\[
u(X)=g(X,U), \qquad v(X)=g(X,V).
\]
Define the symmetric bilinear forms
\[
S_{-}(X,Y)
   =u(X)u(Y)-v(X)v(Y),
\]
and
\[
S_{+}(X,Y)
   =u(X)v(Y)+v(X)u(Y).
\]

Let $A,B\in C^\infty(\Omega)$  and define a $(1,2)$-tensor $K$ by
\begin{align}
K_XY
&=
\bigl(A S_{-}(X,Y)+B S_{+}(X,Y)\bigr)U
\nonumber\\
&\quad+
\bigl(B S_{-}(X,Y)-A S_{+}(X,Y)\bigr)V .
\label{eq:surface-K}
\end{align}
Since $S_{-}$ and $S_{+}$ are symmetric, we have
\[
K_XY=K_YX.
\]
Moreover, the cubic form
\[
C(X,Y,Z)=g(K_XY,Z)
\]
is totally symmetric. Hence
\[
\nabla=\nabla^{g}+K,
\qquad
\nabla^{*}=\nabla^{g}-K
\]
define dual statistical connections on $(\Omega,g)$.

Define an endomorphism $\Theta$ of $T\Omega$ by
\begin{equation}
\Theta Y=u(Y)V-v(Y)U.
\label{eq:surface-Theta}
\end{equation}
Thus
\[
\Theta U=V,
\qquad
\Theta V=-U,
\qquad
\Theta^{2}=-I.
\]
Since $\Theta$ is the complex structure determined by the metric and
the orientation of the surface, it is parallel with respect to the
Levi-Civita connection:
\[
\nabla^{g}\Theta=0.
\]

Furthermore,
\[
u(\Theta Y)=-v(Y),
\qquad
v(\Theta Y)=u(Y),
\]
and therefore
\[
S_{-}(X,\Theta Y)=-S_{+}(X,Y),
\qquad
S_{+}(X,\Theta Y)=S_{-}(X,Y).
\]
Using \eqref{eq:surface-K}, we obtain
\begin{align*}
K_X(\Theta Y)
&=
\bigl(-A S_{+}(X,Y)+B S_{-}(X,Y)\bigr)U
\\
&\quad+
\bigl(-B S_{+}(X,Y)-A S_{-}(X,Y)\bigr)V,
\end{align*}
whereas
\begin{align*}
\Theta(K_XY)
&=
-\bigl(B S_{-}(X,Y)-A S_{+}(X,Y)\bigr)U
\\
&\quad+
\bigl(A S_{-}(X,Y)+B S_{+}(X,Y)\bigr)V.
\end{align*}
Consequently,
\[
K_X\Theta=-\Theta K_X.
\]
By Proposition~\ref{prop:K-characterization}, $\Theta$ satisfies the
gauge equation
\[
\nabla_X(\Theta Y)=\Theta(\nabla^{*}_X Y).
\]
Thus $(\Omega,g,\nabla,\nabla^*,\Theta)$ is a
$\Theta$-statistical manifold. Hence, locally on every oriented Riemannian surface, the construction
gives a two-function family of gauge-compatible statistical
structures, parametrized by arbitrary smooth functions $A$ and $B$.
The resulting statistical structure is nontrivial whenever $A$ and
$B$ are not both identically zero.

\subsection{A non-totally-geodesic proper $\Theta$-CR submanifold}

Let $\widetilde M=\mathbb{R}^{4}$ be endowed with its Euclidean metric
and the standard complex structure $J$ defined by
\[
Je_{1}=e_{2},\qquad Je_{2}=-e_{1},
\qquad
Je_{3}=e_{4},\qquad Je_{4}=-e_{3}.
\]
Let $a,b\in\mathbb{R}$ with $(a,b)\neq(0,0)$ and define the statistical difference tensor
$\widetilde K$ by
\begin{align*}
\widetilde K_{e_{1}}e_{1}
   &=a e_{1}+b e_{2},\\
\widetilde K_{e_{1}}e_{2}
   &=\widetilde K_{e_{2}}e_{1}
     =b e_{1}-a e_{2},\\
\widetilde K_{e_{2}}e_{2}
   &=-a e_{1}-b e_{2},
\end{align*}
and
\[
\widetilde K_XY=0
\]
whenever at least one of $X,Y$ belongs to
$\operatorname{span}\{e_{3},e_{4}\}$.

The tensor $\widetilde K$ is symmetric, its associated cubic form is
totally symmetric, and
\[
\widetilde K_XJ=-J\widetilde K_X.
\]
Since $J$ is parallel with respect to the Euclidean Levi-Civita
connection, the connections
\[
\widetilde{\nabla}
   =\widetilde{\nabla}^{g}+\widetilde K,
\qquad
\widetilde{\nabla}^{*}
   =\widetilde{\nabla}^{g}-\widetilde K
\]
make
\[
(\widetilde M,\widetilde g,
 \widetilde{\nabla},\widetilde{\nabla}^{*},J)
\]
a $J$-statistical manifold.

Fix $r>0$ and consider the immersion
\[
\psi:\mathbb{R}^{2}\times S^{1}\longrightarrow\mathbb{R}^{4},
\qquad
\psi(u,v,t)
   =(u,v,r\cos t,r\sin t).
\]
Along $M=\psi(\mathbb{R}^{2}\times S^{1})$, set
\[
E_{1}=e_{1},
\qquad
E_{2}=e_{2},
\qquad
U=-\sin t\,e_{3}+\cos t\,e_{4}.
\]
Then $\{E_{1},E_{2},U\}$ is a local orthonormal tangent frame. A unit
normal vector field is
\[
N=\cos t\,e_{3}+\sin t\,e_{4}.
\]
The complex structure satisfies
\[
JE_{1}=E_{2},
\qquad
JE_{2}=-E_{1},
\qquad
JU=-N,
\qquad
JN=U.
\]
Hence
\[
D_{J}=\operatorname{span}\{E_{1},E_{2}\},
\qquad
D_{J}^{\perp}=\operatorname{span}\{U\}
\]
give an orthogonal decomposition
\[
TM=D_{J}\oplus D_{J}^{\perp},
\]
with
\[
J(D_{J})=D_{J},
\qquad
J(D_{J}^{\perp})\subset T^{\perp}M.
\]
Thus $M$ is a proper $J$-CR submanifold.

Moreover,
\[
\ker(P_J|_{TM})
   =D_J^{\perp}
   =\operatorname{span}\{U\},
\]
since $P_J$ restricts to the complex structure on
$D_J=\operatorname{span}\{E_1,E_2\}$. Thus the kernel assumption
in Proposition~3.6(2) is satisfied.

Since $\widetilde K$ vanishes whenever one of its arguments belongs to
$\operatorname{span}\{e_{3},e_{4}\}$, the statistical and Riemannian
second fundamental forms agree in the $U$-direction. A direct
calculation gives
\[
\widetilde\nabla^{g}_{U}U=-\frac{1}{r}N,
\]
and consequently
\begin{equation}
h(U,U)=h^{*}(U,U)=-\frac{1}{r}N.
\label{eq:cylinder-second-fundamental}
\end{equation}
In particular,
\[
h(U,U)\neq0,
\]
so $M$ is not totally geodesic.
Since the tangential component of
$\widetilde\nabla^{*}_{U}U$ vanishes, we also have
\[
\nabla^{*}_{U}U=0.
\]
Hence the one-dimensional distribution $D_J^{\perp}$ is
$\nabla^{*}$-autoparallel.

The Weingarten formula also gives
\[
\widetilde\nabla_{U}N=\frac{1}{r}U,
\]
and hence
\[
A_{N}U=-\frac{1}{r}U.
\]
Since
\[
F_{J}U=JU=-N,
\]
we have
\[
A_{F_{J}U}U
   =A_{-N}U
   =\frac{1}{r}U.
\]
Moreover, from $JN=U$,
\[
t_{J}\bigl(h^{*}(U,U)\bigr)
   =
t_{J}\left(-\frac{1}{r}N\right)
   =
-\frac{1}{r}U.
\]
Therefore
\[
A_{F_{J}U}U
+
t_{J}\bigl(h^{*}(U,U)\bigr)
=
\frac{1}{r}U-\frac{1}{r}U
=0.
\]
This is precisely the identity in Proposition~\ref{prop:CR-autoparallel}
for the anti-invariant distribution. Notice that the two terms are
individually nonzero. Thus the gauge-compatible submanifold identity
is realized through a genuine cancellation between the shape
operator and the dual second fundamental form, rather than through
total geodesicity.

\end{document}